\documentclass{amsart}

\usepackage{enumerate}
\usepackage{amssymb, amsmath}
\usepackage{mathrsfs}
\usepackage[active]{srcltx}
\usepackage{verbatim}
\usepackage{color}

\usepackage{newsymbol}

\let\mathcal \undefined
\def\mathcal{\mathscr}

\let\emptyset \undefined
\let\ge       \undefined
\let\le       \undefined
\newsymbol\le          1336  
\newsymbol\ge          133E  
\newsymbol\emptyset    203F
\newsymbol\notle       230A
\newsymbol\notge       230B

\theoremstyle{plain}
\newtheorem{theorem}{Theorem}[section]
\newtheorem{corollary}[theorem]{Corollary}
\newtheorem{lemma}[theorem]{Lemma}
\newtheorem{proposition}[theorem]{Proposition}

\theoremstyle{remark}
\newtheorem{remark}[theorem]{Remark}

\newtheorem{example}[theorem]{Example}
\newtheorem{assumption}[theorem]{Assumption}

\numberwithin{equation}{section}

\def\R{{\mathbb R}}
\def\C{{\mathbb C}}

\newcommand{\E}{{\mathbb E}}

\newcommand{\F}{{\mathscr F}}

\newcommand{\e}{\varepsilon}

\newcommand{\eps}{\varepsilon}

\newcommand{\beq}{\begin{equation}}
\newcommand{\eeq}{\end{equation}}
\newcommand{\bal}{\begin{aligned}}
\newcommand{\eal}{\end{aligned}}
\newcommand{\ben}{\begin{enumerate}}
\newcommand{\beni} {\begin{enumerate}[(i)]}
\newcommand{\een}{\end{enumerate}}
\newcommand{\bit}{\begin{itemize}}
\newcommand{\eit}{\end{itemize}}
\newcommand{\beqw}{\begin{equation*}}
\newcommand{\eeqw}{\end{equation*}}
\newcommand{\bthm}{\begin{theorem}}
\newcommand{\ethm}{\end{theorem}}
\newcommand{\bpr}{\begin{proposition}}
\newcommand{\epr}{\end{proposition}}
\newcommand{\ble}{\begin{lemma}}
\newcommand{\ele}{\end{lemma}}
\newcommand{\blem}{\begin{lemma}}
\newcommand{\elem}{\end{lemma}}
\newcommand{\bpf}{\begin{proof}}
\newcommand{\epf}{\end{proof}}
\newcommand{\bex}{\begin{example}}
\newcommand{\eex}{\end{example}}
\newcommand{\bre}{\begin{example}}
\newcommand{\ere}{\end{example}}
\newcommand{\bma}{\begin{bmatrix}}
\newcommand{\ema}{\end{bmatrix}}

\newcommand{\Dom}{{\mathsf D}}
\newcommand{\Ran}{{\mathsf R}}
\newcommand{\Ker}{{\mathsf N}}

\newcommand{\calL}{{\mathscr L}}

\newcommand{\n}{\Vert}
\newcommand{\one}{{{\bf 1}}}
\newcommand{\embed}{\hookrightarrow}
\newcommand{\s}{^*}
\newcommand{\lb}{\langle}
\newcommand{\rb}{\rangle}

\newcommand{\ov}{\overline}

\begin{document}

\title[$L^p$-Poincar\'e inequality for Ornstein-Uhlenbeck operators]{The $L^p$-Poincar\'e inequality for analytic 
Ornstein-Uhlenbeck operators}

\author{Jan van Neerven}

 \begin{abstract} Consider the linear stochastic evolution equation
$$ dU(t) = AU(t)\,dt + \,dW_H(t), \quad t\ge 0,$$
where $A$ generates a $C_0$-semigroup on a Banach space $E$ and $W_H$ is a cylindrical Brownian motion
in a continuously embedded Hilbert subspace $H$ of $E$. Under the assumption that the solutions to 
this equation admit an invariant measure
$\mu_\infty$ we prove that if the associated Ornstein-Uhlenbeck semigroup is analytic and has compact resolvent,  
then the Poincar\'e inequality 
$$ \n f - \ov f\n_{L^p(E,\mu_\infty)} \le \n D_H f\n_{L^p(E,\mu_\infty)} $$
holds for all $1<p<\infty$. Here 
$\ov f$ denotes the average of $f$ with respect to $\mu_\infty$ and 
$D_H$ the Fr\'echet derivative in the direction of $H$.
 \end{abstract}



\date\today

\thanks{The author gratefully acknowledges financial support by VICI subsidy 639.033.604
of the Netherlands Organisation for Scientific Research (NWO)}

\maketitle

\section{Introduction}

 Let $E$ be a real Banach space and let $H$ be a Hilbert subspace of $E$, with continuous embedding  $i: H\embed E$.
Let $A$ be the generator of a $C_0$-semigroup $S = (S(t))_{t\ge 0}$ on $E$ and let 
$W_H$ be a cylindrical Brownian motion in $H$. 
Under the assumption that the linear stochastic evolution equation
\begin{equation}\label{eq:SCP} dU(t) = AU(t) + \,dW_H(t), \quad t\ge 0,\end{equation}
has an invariant measure $\mu_\infty$,  
we wish to establish sufficient conditions for the validity of the Poincar\'e inequality
$$ \n f - \ov f\n_{L^p(E,\mu_\infty)} \le C \n D_H f\n_{L^p(E,\mu_\infty)}, \quad 1<p<\infty.$$
Here $\ov f$ denotes the average of $f$ with respect to $\mu_\infty$ and 
$D_H$ the directional Fr\'echet derivative in the direction of $H$ (see \eqref{eq:DH} below).
To the best of our knowledge, this problem has been considered so far only for $p=2$ and Hilbert spaces $E$.
For this setting,
Chojnowska-Michalik and Goldys \cite{CMG} obtained various necessary and sufficient
conditions for the inequality to be true. 
Here we show that these conditions are equivalent to another, formally weaker, condition and that 
these equivalent conditions imply the validity of the Poincar\'e inequality for all $1<p<\infty$ (Theorem \ref{thm:poincare-main}). 
Our proof depends crucially on the $L^p$-gradient estimates 
for analytic Ornstein-Uhlenbeck semigroups obtained in the recent papers \cite{MaaNee09,MaaNee11}.

Related $L^p$-Poincar\'e inequalities have been proved in various other settings,
e.g. for the classical Ornstein-Uhlenbeck semigroup
(this corresponds to the case $A=-I$ of the setting considered here) \cite[Eq. (2.5)]{Pis}, 
for the Walsh system \cite{EfrLP}, and in certain non-commutative situations \cite{JunZen, Zen}.
Poincar\'e inequalities are intimately related to other functional inequalities such as,
log-Sobolev inequalities and transportation cost inequalities, and imply concentration-of-measure 
inequalities. For a comprehensive study of these topics we refer the reader to the recent
monograph of Bakry, Gentil and Ledoux \cite{BGL}.

As an application of Theorem \ref{thm:poincare-main} we find that the $L^p$-Poincar\'e inequality holds 
if the Ornstein-Uhlenbeck semigroup $P$ associated with \eqref{eq:SCP} (see \eqref{eq:P}) is analytic on $L^p(E,\mu_\infty)$ and 
has compact resolvent. 
In Section \ref{sec:ex} we provide some examples
in which the various assumptions are satisfied. In the final Section \ref{sec:comp} we address 
the problem of compactness of certain tensor products of resolvents naturally associated with $P$. 
  
All vector spaces are real. We will always identify Hilbert spaces with their dual via the Riesz representation theorem. 
The domain, kernel, and range of a linear operator $A$ will be denoted by $\Dom(A)$, $\Ker(A)$, and $\Ran(A)$, respectively.
We write $a\lesssim b$ to mean that there exists a constant $C$, independent of $a$ and $b$, such that $a\le Cb$.

\section{The $L^p$-Poincar\'e inequality}\label{sec:main}

 Throughout this note we fix a Banach space $E$ and a Hilbert subspace $H$ of $E$,
  with continuous embedding  $i: H\embed E$, and make the following assumption.

\begin{assumption}\label{ass:mu-infty}
 There exists a centred Gaussian Radon measure $\mu_\infty$ on $E$ whose covariance operator $Q_\infty\in\calL(E\s,E)$ is given by 
$$\lb Q_\infty x\s, y\s\rb = \int_0^\infty \lb Q S\s(s)x\s, S\s(s) y\s\rb \,ds, \quad x\s,y\s\in E\s.$$
\end{assumption}

\medskip
\noindent Here $Q:= i\circ i\s$; we identify $H$ and its dual in the usual way.
The convergence of the integrals on the right-hand side is part of the assumption. As is well known,
Assumption \ref{ass:mu-infty} is equivalent to the existence of an invariant measure for the problem
\eqref{eq:SCP}; 
we refer the reader to \cite{DaPZab, GolNee} for the details. In fact, the measure $\mu_\infty$ is the minimal (in the sense
of covariance domination) invariant measure for \eqref{eq:SCP}.

The formula
\begin{align}\label{eq:P} P(t)f(x) = \E (f(U(t,x))),\quad t\ge 0, \ x\in E,
\end{align}
where $U(t,x)$ denotes the unique mild solution of \eqref{eq:SCP} with initial value $x$,
defines 
a semigroup of linear contractions $P = (P(t))_{t\ge 0}$ 
on the space $B_{\rm b}(E)$ of bounded real-valued Borel functions on $E$.
This semigroup 
is called the {\em Ornstein-Uhlenbeck semigroup} associated with the pair $(A,H)$. By an easy application of H\"older's inequality, 
this semigroup extends uniquely to   
$C_0$-semigroup of contractions on $L^p(E,\mu_\infty)$, which we shall also denote by $P$.
Its generator will be denoted by $L$.

By a result of Chojnowska-Michalik and Goldys \cite{ChojGol96, CMG} (see \cite{Nee98} for the formulation of this result in its
present generality), the reproducing kernel Hilbert space
$H_\infty$ associated with the measure $\mu_\infty$ is invariant under the semigroup $S$ and the restriction of $S$ 
is a $C_0$-semigroup
of contractions on $H_\infty$. We shall denote this restricted semigroup by $S_\infty$ and its generator by $A_\infty$.
The inclusion mapping $H_\infty\embed E$ will be denoted by $i_\infty$; recall that $Q_\infty = i_\infty\circ~i_\infty\s$
(see \cite{GolNee, Nee98}).

It has been shown in \cite{ChojGol96} (see also \cite{Nee98, Nee05}) that $P(t)$ is the so-called 
{\em second quantisation}
of the adjoint semigroup $S_\infty\s(t)$. More precisely, the Wiener-It\^o isometry establishes an isometric identification 
$L^2(E,\mu_\infty) = \bigoplus_{n\ge 0}H_\infty^{{\hbox{\tiny\textcircled{s}}}n}$,
where $H_\infty^{\hbox{\tiny\textcircled{s}}n}$ is the $n$-fold symmetric tensor product of $H_\infty$
(the so-called $n$-th Wiener-It\^o chaos), and under this isometry
we have $$ P(t) = \bigoplus_{n\ge 0} S_\infty^{*{\hbox{\tiny\textcircled{s}}n}}(t).$$
We have $H_\infty^{\hbox{\tiny\textcircled{s}}0} = \R\one$ (by definition) and $H_\infty^{\hbox{\tiny\textcircled{s}}1} = H_\infty$. 
The latter identification allows us to deduce many properties of $P$ from the corresponding properties of
$S_\infty\s$ and vice versa and will be used freely in what follows.

Following \cite{CerGoz, GolNee} we define
$\F^k$ as the space of all functions $f:E\to\R$ of the 
form
\begin{equation}
\label{eq:cyl}
f(x) = \phi(\lb x,x_1\s\rb, \hdots,\lb x,x_d\s\rb)
\end{equation}
for some $d\ge 1$, 
with $x_j\s\in E\s$ for
all $j=1,\hdots,d$ with $\phi\in C_{\rm b}^k(\R^d)$.
Let 
$$\F_A^k = \{f\in \F^k: \ \hbox{$x_j\s\in\Dom(A\s)$ for
all $j=1,\hdots,d$ and }\lb\, \cdot\,, A\s Df(\cdot)\rb\in C_{\rm b}(E)\}.$$
It follows from \cite{CerGoz, GolNee} that $\F_A^2$ is a core for $\Dom(L)$ in each $L^p(E,\mu_\infty)$
 and that for $f,g\in \F_A^2$ we have the identity
\begin{align}\label{eq:L-form} 
\lb Lf,g\rb  + \lb Lg,f\rb = -\int_E \lb D_H f, D_H g\rb\,d\mu_\infty.
\end{align}
Here $D_H$ denotes the Fr\'echet derivative in the direction of $H$, defined on $\F^1$ 
by
\begin{align}\label{eq:DH} D_H f(x) := \sum_{j=1}^n \frac{\partial \phi}{\partial x_j}(\lb
x,x_1\s\rb,\dots,\lb x,x_n\s\rb) \, i\s x_j\s
\end{align}
with $f$ and $\phi$ as in \eqref{eq:cyl}. It should be emphasised that $D_H$ is not always closable; 
various conditions for closability as well as a counterexample are given in \cite{GGN}.
If $P$ is analytic on $L^p(E,\mu_\infty)$ for some/all $1<p<\infty$ (the equivalence
being a consequence of the Stein interpolation theorem), then $D_H$ is
closable as an operator from $L^p(E,\mu_\infty)$ to $L^p(E,\mu_\infty;H)$
\cite[Proposition 8.7]{GolNee}.

The following necessary and sufficient condition for the $L^2$-Poincar\'e inequality is
essentially due to Chojnowska-Michalik and Goldys \cite{ChojGol99}
(see also \cite[Proposition 10.5.2]{DPZ02}). Since the present formulation is slightly more general,
for the convenience we include the proof which follows the lines of  \cite{ChojGol99}.

\begin{proposition}[Poincar\'e inequality, the case $p=2$]\label{prop:poincare} 
Let Assumption \ref{ass:mu-infty} hold and fix a number $\omega>0$. 
If $D_H$ is closable as a densely defined operator in $L^2(E,\mu_\infty)$, then the following assertions are equivalent:
\begin{enumerate}[\rm(1)]
 \item $\n S_\infty(t)\n \le e^{-\omega t}$ for all $t\ge 0$;
 \item The Poincar\'e inequality  
$$ \n f - \ov f\n_{L^2(E,\mu_\infty)} \le \frac1{\sqrt{2\omega}}\n D_H f\n_{L^2(E,\mu_\infty)} , \quad f\in\Dom(D_H),$$
holds. Here, $\overline f = \int_E f\,d\mu_\infty$.
\end{enumerate}
\end{proposition}

\begin{proof}
(1)$\Rightarrow$(2): 
Since $t\mapsto e^{\omega t}S_\infty\s(t)$ is a $C_0$-contraction semigroup, by second quantisation
the same is true for the direct sum for $n\ge 1$ of their $n$-fold symmetric tensor products,  
$ \bigoplus_{n\ge 1} e^{n\omega t} S_\infty^{*{\hbox{\tiny\textcircled{s}}n}}(t).$
Replacing $e^{n\omega t}$ by $e^{\omega t}$,
the resulting direct sum $ \bigoplus_{n\ge 1} e^{\omega t} S_\infty^{*{\hbox{\tiny\textcircled{s}}n}}(t)$
is contractive as well. This semigroup is generated by the part $L_0+\omega$ of $L+\omega$ in 
$L_0^2(E,\mu_\infty):= L^2(E,\mu_\infty)\ominus\R\one$. Thus
we obtain the dissipativity inequality $$-\lb (L_0+\omega)f,f\rb\ge 0, \quad f\in \Dom(L_0).$$ 
In view of \eqref{eq:L-form}, this gives the inequality 
$$\omega\n f\n_2^2 \le -\lb L_0 f,f\rb =  \frac12\n D_H f\n_2^2, \quad f\in \Dom(L_0)\cap\F_A^2.$$
As a consequence,
\begin{align}\label{eq:Poinc-2} 
\omega\n f - \ov f\n_2^2 \le  \frac12\n D_H f\n_2^2, \quad f\in \F_A^2.
\end{align}

It is routine (albeit somewhat tedious) to check that the inequality \eqref{eq:Poinc-2}
extends to $f\in \F^1$, and since by definition this 
is a core for $\Dom(D_H)$ it extends to arbitrary elements $g\in \Dom(D_H)$.

\smallskip 
(2)$\Rightarrow$(1):  
Every $x\s\in E\s$, when viewed as an element of $L^2(E,\mu_\infty)$, satisfies
$ D_H x\s = i\s x\s$. Moreover,  if $x\s\in \Dom(A\s)$, then $A_\infty\s x\s\in \Dom(A_\infty\s)$,
and therefore (identifying $H_\infty$ with the first Wiener-It\^o chaos) $ x\s\in \Dom(L)$ as an element of $L^2(E,\mu_\infty)$. 

By specialising the Poincar\'e inequality to functionals $x\s$ we obtain the inequality
$$ \n i_\infty\s x\s\n = \n x\s\n_{L^2(E,\mu_\infty)} \le \frac1{\sqrt{2\omega}} \n i\s x\s\n, 
\quad x\s\in \Dom(A\s).$$
In the same way, \eqref{eq:L-form} takes the form
$$ \lb A_\infty\s i_\infty\s x\s,i_\infty\s x\s\rb =  -\frac12 \n i\s x\s\n^2, \quad x\s\in \Dom(A\s).$$
Combining these inequalities, we obtain
$$ -\lb A_\infty\s i_\infty\s x\s ,i_\infty\s x\s\rb \ge \omega\n i_\infty\s x\s\n^2, \quad x\s\in \Dom(A\s).$$
Since the elements $i_\infty\s x\s$ with $x\s\in \Dom(A\s)$ form a core for $\Dom(A_\infty\s)$,   
this is equivalent to saying that $A_\infty\s + \omega$ is dissipative on $H_\infty$. It follows that
$\n S_\infty\s(t)\n \le \exp(-\omega t)$ for all $t\ge 0.$
\end{proof}

The main result of this note (Theorem \ref{thm:poincare-main})  asserts that if $P$ is analytic
and $A_\infty\s$ has closed range, 
then all conditions of
Proposition \ref{prop:poincare} are satisfied 
and the Poincar\'e inequality extends to $L^p(E,\mu_\infty)$ for all $1<p<\infty$. 
To prepare for the proof we need to recall some preliminary facts.
We begin by imposing the following assumption, which will be in force for the rest of this section.

\begin{assumption}\label{ass:L-analytic}
For some (equivalently, for all) $1<p<\infty$ the semigroup $P$ extends to an analytic $C_0$-semigroup on $L^p(E,\mu_\infty)$. 
\end{assumption}
The problem of analyticity of $P$ has been studied in several articles \cite{Fuh, Gol, GolNee, MaaNee07, MaaNee11}.
In these, necessary and sufficient conditions for analyticity can be found. 
We have already mentioned the fact that 
if $P$ is analytic on $L^p(E,\mu_\infty)$ for some/all $1<p<\infty$, then $D_H$ is
closable as an operator from $L^p(E,\mu_\infty)$ to $L^p(E,\mu_\infty;H)$. 
In what follows, $D_H$ will always denote this closure and
$\Dom(D_H)$ its domain in $L^p(E,\mu)$. Note that there is a slight abuse of notation here,
as $\Dom(D_H)$ obviously depends on $p$. The choice of $p$ will always be clear from the context,
and for this reason we prefer not to overburden notations. The same slight abuse of notation
applies to the notation $\Dom(L)$ for the domain of $L$ in $L^p(E,\mu_\infty)$.

From \cite{MaaNee07} we know that if $P$ is analytic, then the generator $L$ of $P$ can be represented as
\begin{align}\label{eq:L} 
L = D_H\s B D_H 
\end{align}
for a unique bounded operator $B$ on $H$ which satisfies $$B+B\s = -I.$$ 
The rigorous interpretation of \eqref{eq:L} is that for $p=2$ the operator $-L$ is the sectorial operator
associated with the closed continuous accretive form
$$ (f,g)\mapsto -\lb BD_H f, D_H g\rb.$$

In the sequel we will use the standard fact (which is proved by hypercontractivity arguments)
that for each $n\ge 0$ the summand $H_\infty^{\hbox{\tiny\textcircled{s}}n}$ in the Wiener-It\^o
decomposition for $L^2(E,\mu_\infty)$ is contained as a closed subspace in $L^p(E,\mu_\infty)$ for all
$1<p<\infty$. In view of this we will continue to refer to $H_\infty^{\hbox{\tiny\textcircled{s}}n}$
as the $n$-th Wiener chaos.
By an interpolating argument (see \cite[Lemma 4.2]{Nee05})
we obtain the estimate $\n P(t)\n_p \le \n S_\infty(t)\n^{n\theta_p}$ on each of these subspaces, with a constant $0<\theta_p<1$ depending
only on $p$. Summing over $n\ge 1$ and passing to the closure of the linear span, we obtain the estimate
\begin{align}\label{eq:ues} 
\n P(t)\n_p \le \n S_\infty(t)\n^{\theta_p} \ \hbox{ on $L^p(E,\mu_\infty)\ominus \R\one$}. 
\end{align}

\begin{theorem}[$L^p$-Poincar\'e inequality]\label{thm:poincare-main}
Let Assumptions \ref{ass:mu-infty} and  \ref{ass:L-analytic} hold. 
Then the following assertions are equivalent:
\begin{enumerate}[\rm(1)]
 \item $A_\infty\s$ has closed range;
 \item there exists $\omega>0$ such that $\n S_\infty(t)\n \le e^{-\omega t}$ for all $t\ge 0$;
 \item there exist $M\ge 1$ and $\omega>0$ such that $\n S_\infty(t)\n \le Me^{-\omega t}$ for all $t\ge 0$; 
 \item there exist $M\ge 1$ and $\omega>0$ such that $\n S_H(t)\n \le Me^{-\omega t}$ for all $t\ge 0$;
 \item $H_\infty$ embeds continuously in $H$;
 \item for some $1<p<\infty$ there exists a finite constant $C\ge 0$ such that  
$$ \n f - \ov f\n_{L^p(E,\mu_\infty)} \le C_p\n D_H f\n_{L^p(E,\mu_\infty)},\quad  f\in \Dom(D_H);$$ 
 \item for all $1<p<\infty$ there exists a finite constant $C\ge 0$ such that  
$$ \n f - \ov f\n_{L^p(E,\mu_\infty)} \le C_p\n D_H f\n_{L^p(E,\mu_\infty)},\quad  f\in \Dom(D_H).$$  
\end{enumerate}
\end{theorem}
In what follows we will say that {\em the $L^p$-Poincar\'e inequality holds} 
if condition (7) is satisfied.

Before we start with the proof we recall some further useful facts.
Firstly, on the first Wiener chaos, \eqref{eq:L} reduces to the identity
$$ A_\infty\s  = V\s BV,$$
where $V$ is the closure of the mapping $i_\infty\s x\s\mapsto i\s x\s$; see \cite{GGN, MaaNee09, MaaNee11}.
Secondly, in \cite{MaaNee11} it is shown that Assumption \ref{ass:L-analytic} implies that 
$S$ maps $H$ into itself and that its restriction to $H$ extends to a bounded analytic $C_0$-semigroup on $H$. 
We shall denote this semigroup by $S_H$ and its generator by $A_H$.

\begin{proof}[Proof of Theorem \ref{thm:poincare-main}]
(1)$\Rightarrow$(3): Let us first observe that the strong stability of $S_\infty\s$ \cite[Proposition 2.4]{GolNee} 
implies that $\Ker(A_\infty\s) = \{0\}$.

Suppose next that some
$h\in H_\infty$ annihilates the range of $A_\infty\s$. 
As $\lb A_\infty\s g,h\rb = \lb V\s B V g, h\rb = 0$ for all $g\in\Dom(A_\infty\s)$, it follows that $h\in \Dom(V)$ and 
$\lb BV\s g, Vh\rb = 0$ for all $g\in \Dom(A_\infty\s)$. Using that $\Dom(A_\infty\s) $ is a core for 
$\Dom(V)$ (see \cite{MaaNee09}), it follows that $\lb BV\s g, Vh\rb = 0$ for all $g\in \Dom(V)$.
In particular, $\lb BV\s h, Vh\rb = 0$. Since also $\lb BV\s h, Vh\rb =-\frac12\n V h\n^2$ by the identity 
$B+B\s = -I$, it follows that $Vh=0$ and therefore $h\in \Ker (A_\infty\s) = \Ker (V)$.
But we have already seen that $\Ker(A_\infty\s) = \{0\}$ and we conclude that $h=0$. 

This argument proves that $\Ran(A_\infty\s)$ is dense. 
Since by assumption $A_\infty\s$
has closed range,
it follows that $A_\infty\s$ is surjective. As we observed at the beginning of the proof, $A_\infty\s$ is also
injective, and therefore $A_\infty\s$ is boundedly invertible by the closed graph theorem. Since $A_\infty\s$ generates an analytic $C_0$-contraction semigroup, 
 the spectral mapping theorem for analytic $C_0$-semigroups (see \cite{EngNag}) implies that $S_\infty\s$ is uniformly exponentially 
 stable.

(3)$\Rightarrow$(7): 
Fix an arbitrary $1<p<\infty$.
Fix a function $f\in \F^0$ and let $\frac1p+\frac1q$.
Then 
$$ \n f-\ov f\n 
= \sup_{\substack{\n g\n_q\le 1\\ \ov g = 0}} 
 |\lb f - \overline f,g\rb|
= \sup_{\substack{\n g\n_q\le 1\\ \ov g = 0}}   |\lb f - \overline f,g-\ov g\rb|= \sup_{\substack{\n g\n_q\le 1\\ \ov g = 0}}   |\lb f ,g-\ov g\rb| ,
$$
where it suffices to consider functions $g\in \F^0$. Next we observe that, by \eqref{eq:ues}, 
$$ \lb f,g-\overline g\rb
= \lim_{t\to\infty} \lb f, g- P(t)g\rb.$$
Following an argument in \cite[Lemma 3]{Led} we have
\begin{align*}  \lb f, g- P(t)g\rb & = -\int_0^t \lb f, LP(s) g\rb\,ds
= -\int_0^t \lb D_H f, BD_H P(s) g\rb \,ds.
\end{align*}
If in addition $\ov g=0$ (i.e. if $g\in L^p(E,\mu_\infty)\ominus\R\one$), then for all $t\ge 1$ we have
\begin{align*} |\lb f, g-P(t)g\rb| 
& \le \n B\n \n D_H f\n_p \Big(\int_0^1 + \int_1^\infty \Big)\n D_H P(s) g\n_q \,ds 
\\ & \lesssim  \n D_H f\n_p\Big( \int_0^1 \frac1{\sqrt s}\n g\n_q \,ds + \n D_H P(1)\n \int_0^\infty e^{-\omega\theta_q}\n g\n_q\,ds\Big).
\end{align*}
where we used the gradient estimates of \cite{MaaNee09} and \eqref{eq:ues}.
Taking the supremum over all $g\in \F^0$  
of $L^q$-norm $1$ with $\ov g = 0$,  this gives 
$$ \n f - \overline f\n_p \lesssim \n D_H f\n_p.$$
Since $\F^0$ is a core for $\Dom(D_H)$ 
this concludes the proof of the implication. 

(7)$\Rightarrow$(6): This implication is trivial. 

(6)$\Rightarrow$(3): This follows from Proposition \ref{prop:poincare}
along with the fact that $H_\infty$ is isomorphic to the first Wiener-It\^o chaos in $L^p(E,\mu_\infty)$.

(3)$\Rightarrow$(1): The uniform exponential stability of $S_\infty\s$ implies that 
$A_\infty\s$ is boundedly invertible.

(3)$\Leftrightarrow$(4)$\Leftrightarrow$(5): These equivalences have been proved in \cite[Theorem 5.4]{GolNee}.

(7)$\Rightarrow$(2): This follows from Proposition \ref{prop:poincare}.

(2)$\Rightarrow$(3): Trivial.
\end{proof}
The equivalent conditions of the theorem do not in general imply the existence 
of an $\omega>0$ such that $\n S_H(t)\n \le e^{-\omega t}$ for all $t\ge 0$:

\begin{example}
 Consider the Dirichlet Laplacian $\Delta$ on $E = L^2(-1,1)$ and take $H=E$. Let $S$ denote the 
 heat semigroup generated by $\Delta$ on this space. Fix $\omega>0$. 
 As is well known and easy to check,
 Assumptions \ref{ass:mu-infty} and  \ref{ass:L-analytic} are satisfied for the operator $\Delta-\omega$.
 Let us now replace the norm of $L^2(-1,1)$ by the equivalent (Hilbertian) norm
 $$ \n f\n_{(r)}^2 := \n f|_{(-1,0)}\n^2 + r^2 \n f|_{(0,1)}\n^2,$$
 where $r>0$ is a positive scalar. Starting from an initial condition with support in $(-1,0)$, the 
 semigroup $s_\omega(t) = e^{-\omega t}S(t)$ generated by $\Delta-\omega$ will instantaneously spread out the support of $f$ over the
 entire interval $(-1,1)$. Hence if we fix $t_0>0$ and $\omega>0$ we may choose $r_0>0$ so large that 
 $$ \n S_\omega(t_0)f\n_{(r)} > \n f\n_{(r)}.$$
 As a result, the semigroup  $S_\omega$ is uniformly exponentially stable but not contractive
 on $L^2(-1,1)$ endowed with the norm $ \n \cdot\n_{(r_0)}$.
 \end{example}
One could object to this example that there is an equivalent Hilbertian norm (namely, the original norm of $L^2(-1,1)$)
on which we do have $\n S_\omega(t)\n \le e^{-\omega t}$. There exist examples, however, of bounded analytic 
Hilbert space semigroups which are not similar to an analytic contraction semigroup. Such examples may be realised
as multiplication semigroups on a suitable (pathological) Schauder basis (see, e.g., \cite{LeM} and the references given there). 
For such examples,
Assumptions \ref{ass:mu-infty} and  \ref{ass:L-analytic} are again satisfied and we obtain a counterexample
that cannot be repaired by a Hilbertian renorming.

As an application of Theorem \ref{thm:poincare-main} we have the following sufficient 
condition for the validity of the $L^p$-Poincar\'e inequality.

\begin{theorem}[Compactness implies the $L^p$-Poincar\'e inequality]\label{thm:poincare-compact} Let Assumptions \ref{ass:mu-infty} and  \ref{ass:L-analytic} hold
and fix $1<p<\infty$. 
The following assertions are equivalent:
\begin{enumerate}[\rm(1)]
\item $L$ has compact resolvent on $L^p(E,\mu_\infty)$;
\item $P$ is compact on $L^p(E,\mu_\infty)$;
\item $A_\infty$ has compact resolvent on $H_\infty$;
\item $S_\infty$ is compact on $H_\infty$;
\item $A_H$ has compact resolvent on $H$; 
\item $S_H$ is compact on $H$. 
\end{enumerate}
If these equivalent conditions are satisfied, then the $L^p$-Poincar\'e inequality holds
 for all $1<p<\infty$.
\end{theorem}
\begin{proof} 
The equivalences
(1)$\Leftrightarrow$(2),  
(3)$\Leftrightarrow$(4), and (5)$\Leftrightarrow$(6) follow from 
\cite[Theorem 4.29]{EngNag} since $P$, $S_\infty$, and $S_H$ are analytic semigroups.

We will prove next that (4) implies 
the validity of the $L^p$-Poincar\'e inequality. 
We will use some elementary facts from semigroup theory 
which can all be found in \cite{EngNag}. 
The strong stability of $S_\infty\s$ implies that $1$ is not an eigenvalue of $S_\infty\s(t)$ for any $t>0$.
Since these operators are compact it follows that $1\not\in \sigma(S_\infty\s(t))$, which in turn 
implies that $0\not\in \sigma(A_\infty\s)$ by the spectral mapping theorem for eventually norm continuous
semigroups. By the equality spectral bound and growth bound for such semigroups, it follows that 
$S_\infty\s$ (and hence also $S_\infty$) is uniformly exponentially stable. We may now 
apply Theorem \ref{thm:poincare-main} to obtain the conclusion.

(2)$\Rightarrow$(4): This follows by restricting to the first Wiener-It\^o chaos.

(4)$\Rightarrow$(2): We have already seen that (4) implies that $S_\infty\s$ is uniformly exponentially stable.
Because of this, the compactness of $S_\infty\s(t)$ implies, by second quantisation, the compactness
 of $P(t)$ on $L^p(E,\mu_\infty)$ (cf. \cite[Lemma 4.2]{Nee05}).

(4)$\Rightarrow$(6): By \cite[Theorem 3.5]{GolNee} combined with \cite[Proposition 1.3]{Nee98},
for each $t>0$ the operator $S(t)$ maps $H$ into $H_\infty$; we shall denote this operator by
$S_{H,\infty}(t)$.
Furthermore we have a continuous embedding
$i_{\infty,H}:H_\infty\embed H$ \cite[Theorem 5.4]{GolNee} (this result can be applied here since, by
what has already been proved, (2) implies the uniform exponential stability of $S_\infty$). Now if 
 $S_\infty$ is compact, the compactness of $S_H$ follows from the factorisation 
 $$S_H(t) = i_{\infty,H}\circ S_\infty(t/2) \circ S_{H,\infty}(t/2).$$ 

(6)$\Rightarrow$(4): We will show that (6) implies that 
$H_\infty$ embeds into $H$. Once we know this, (4) follows from 
the factorisation
$S_\infty(t) = S_{H,\infty}(t/2) \circ S_H(t/2) \circ i_{\infty,H}.$

\smallskip
This concludes the proof of the equivalences of the conditions  (1)--(6). 
To complete the proof we will now show that
these conditions imply the validity of the Poincar\'e inequality.

Suppose that $h\in H$ is a vector satisfying 
$S_H(t)h = h$ for all $t\ge 0$. Since $S(t)$ maps $H$ into $H_\infty$ (see \cite[Proposition 2.3]{GolNee})
this means that $h\in H_\infty$. But then in $E$ for all $t\ge 0$ we have 
$i_\infty S_\infty(t)h = i_H S_H(t)h = i_H h = i_\infty h$, so that in $H_\infty$ we obtain
$S_\infty(t)h=h$ for all $t\ge 0$. Hence, for all $h'\in H_\infty$,
$$ \lb h,h'\rb_{H_\infty} = \lim_{t\to\infty} \lb S_\infty(t)h,h'\rb_{H_\infty}
=  \lim_{t\to\infty} \lb h, S_\infty\s(t)h'\rb_{H_\infty} = 0$$
by the strong stability of $S_\infty\s$. This being true for all $h'\in H_\infty$, it follows that 
$h = 0$. We have thus shown that $1$ is not an eigenvalue of $S_H(t)$. Having arrived at this
conclusion, the argument given above for $S_\infty$ can now be repeated to conclude that 
$S_H$  is uniformly exponentially stable. Now Theorem \ref{thm:poincare-main} 
implies that $H_\infty$ embeds into $H$. 
\end{proof}
\begin{remark} 
The equivalence of (4) and (6) for symmetric Ornstein-Uhlenbeck semigroups follows from \cite[Theorem 2.9]{CGsymm}.
\end{remark}

\begin{corollary}
Let $1<p<\infty$. If the embedding $\Dom(D_H)\embed L^p(E,\mu_\infty)$ is compact, then the $L^p$-Poincar\'e inequality holds.
\end{corollary}
Recall our abuse of notation to denote by $\Dom(D_H)$ and $\Dom(L)$ the domains of closed operators $D_H$ 
and $L$ in $L^p(E,\mu_\infty)$. Necessary and sufficient conditions for the compactness of 
the embedding $\Dom(D_H)\embed L^p(E,\mu_\infty)$ are stated in \cite{GGN}.
\begin{proof} Since $\Dom(L)$ embeds into $\Dom(D_H)$ (see \cite[Theorem 8.2]{MaaNee09}) 
 this is immediate from the previous theorem.
\end{proof}

Our next aim is to show that also an $L^p$-inequality holds for the adjoint operator $D_H\s$. Here we view
$D_H$ as a closed densely defined operator from $L^q(E,\mu_\infty)$ into $ \overline{\Ran(D_H)}$
and $D_H\s$ a closed densely defined operator from $\overline{\Ran(D_H)}$ into $L^p(E,\mu_\infty)$, $\frac1p+\frac1q=1$.
The proof relies on some facts that have been proved in \cite{MaaNee09, MaaNee11}. We start by observing that 
if Assumptions \ref{ass:mu-infty} and  \ref{ass:L-analytic} hold, then the semigroup
$$\underline P(t):= P(t)\otimes S_H\s(t)$$
extends to a bounded analytic $C_0$-semigroup on $L^p(E,\mu_\infty;H)$, $1<p<\infty$.
We will need the fact that on $\overline{\Ran(D_H)}$ 
the generator $\underline L$ of this semigroup is given by 
$$\underline L = D_H D_H\s B;$$ the proof as well as the rigorous interpretation of the right-hand side is given in the
references just quoted.

\begin{theorem}[$L^p$-Poincar\'e inequality for $D_H\s$]\label{thm:poncare-DHs}  
Let Assumptions \ref{ass:mu-infty} and  \ref{ass:L-analytic} hold. 
 If the equivalent conditions of Theorem \ref{thm:poincare-main} are satisfied, 
then there exists a finite constant $C\ge 0$ such that for all $1<p<\infty$ we have 
$$ \n f \n_{L^p(E,\mu_\infty;H)} \le C_p\n D_H\s f\n_{L^p(E,\mu_\infty;H)},\quad  f\in \Dom(D_H\s),$$
where $D_H\s$ is interpreted as explained above.
\end{theorem}
\begin{proof}
We can follow the proof of Theorem \ref{thm:poincare-main},
this time using that for bounded cylindrical functions $f,g\in \overline{\Ran(D_H)}$ we have
\begin{align*}  \lb f, g- \underline P(t)g\rb & = -\int_0^t \lb f, \underline L\underline P(s) g\rb\,ds
= -\int_0^t \lb D_H\s f, D_H\s B \underline P(s) g\rb \,ds.
\end{align*}
For $t\ge 1$ we then have
\begin{align*} |\lb f, g- \underline P(t)g\rb| 
& \le \n B\n \n D_H\s f\n_p \Big(\int_0^1 + \int_1^\infty \Big)\n D_H\s B \underline P(s) g\n_q \,ds 
\\ & \lesssim  \n D_H\s f\n_p\Big( \int_0^1 \frac1{\sqrt s}\n g\n_q \,ds + \n D_H\s B \underline P(1)\n \int_0^\infty e^{-\omega\theta_q}\n g\n_q\,ds\Big),
\end{align*}
this time using the gradient estimates for $D_H\s B$ (cf. the proof of \cite[Proposition 9.3]{MaaNee09}
where resolvents are used instead of the semigroup operators)
and the uniform exponential stability of $\underline P = P\otimes S_H\s$.
The proof can be finished along the lines of Theorem \ref{thm:poincare-main};
this time we use that $\lim_{t\to\infty}\lb f, g- \underline P(t)g\rb = \lb f, g\rb$.
\end{proof}

\section{Examples}\label{sec:ex}

\begin{example}[Finite dimensions and non-degenerate noise]
Suppose that $H = E = \R^d$ and let Assumption \ref{ass:mu-infty} hold. 
Then $H^\infty = \R^d$. Under these assumptions, a result of Fuhrman \cite[Theorem 3.6 and Corollary 3.8]{Fuh} implies that  
Assumption \ref{ass:L-analytic} holds. By finite-dimensionality, 
the conditions of 
Theorems \ref{thm:poincare-main} and \ref{thm:poncare-DHs}  are satisfied. It follows that
the $L^p$-Poincar\'e inequalities for $D_H$ and $D_H\s$ hold for $1<p<\infty$. 
\end{example}

%
 \begin{example}[The self-adjoint case] Suppose that $H=E$ and $S$ is self-adjoint on $E$. 
 Then Assumption \ref{ass:mu-infty} holds if and only if $S$ is uniformly exponentially stable.
 In this situation, by \cite{GolNee} also $S_\infty$ is self-adjoint and uniformly
 exponentially stable, and $P$ is self-adjoint on $L^2(E,\mu_\infty)$. In particular, 
 Assumption \ref{ass:L-analytic} then holds and therefore the equivalent conditions of 
 Theorem \ref{thm:poincare-main} are satisfied. It follows that
 the $L^p$-Poincar\'e inequality holds for $1<p<\infty$.  
 \end{example}

\begin{example}[The strong Feller case] Suppose that Assumptions \ref{ass:mu-infty} and \ref{ass:L-analytic} hold,
and that $P$ is strongly Feller. As is well known, this is equivalent to the condition that for each $t>0$
the semigroup operator
$S(t)$ maps $E$ into the reproducing kernel Hilbert space $H_t$ associated with $\mu_t$,
 the centred Gaussian Radon measure on $E$ whose covariance operator $Q_t\in\calL(E\s,E)$ is given by 
$$\lb Q_t x\s, y\s\rb = \int_0^t \lb Q S\s(s)x\s, S\s(s) y\s\rb \,ds, \quad x\s,y\s\in E\s.$$
These measures exist by a standard covariance domination argument (note that $\lb Q_t x\s,x\s\rb\le \lb Q_\infty x\s,x\s\rb)$.
By \cite{Nee98} we have a contractive embedding $i_{i,\infty}:H_t\embed H_\infty$. 
Then $S_\infty(t) = i_{t,\infty}\circ S(t)\circ i_\infty$, where 
 $i_\infty:H_\infty\embed E$ is the inclusion mapping. The compactness of  $i_\infty:H_\infty\embed E$
(this mapping being $\gamma$-radonifying; see \cite{Nee-Can})
implies that $S_\infty(t)$ is compact for all $t>0$, and by a general result from semigroup theory
 this implies that the resolvent operators $R(\lambda,A_\infty)$ are compact. 
Similarly from
 $S_H(t) = i_{t,\infty}i_{\infty,H}\circ S(t)\circ i_\infty$, where  $i_{\infty,H}: H_\infty\embed H$ is the embedding mapping
(see \cite[Theorem 5.4]{GolNee} for the proof that this inclusion holds under the present assumptions)
it follows that $S_H(t)$ is compact and therefore $R(\lambda,A_H)$ are compact.
It follows that
the $L^p$-Poincar\'e inequalities for $D_H$ and $D_H\s$ hold for $1<p<\infty$. 
\end{example}

\begin{example}[The case $\Dom(A)\embed H_\infty$]
Suppose that Assumptions \ref{ass:mu-infty} and \ref{ass:L-analytic} hold,
and that we have a continuous inclusion $\Dom(A)\embed H_\infty$. Then
$R(\lambda,A_\infty) =  i_A R(\lambda,A) i_\infty$, where $i_\infty:H_\infty\embed E$ and 
$i_A: \Dom(A)\embed H_\infty$ are the inclusion mappings. The compactness of  $i_\infty:H_\infty\embed E$
implies that $R(\lambda,A_\infty)$ is compact. It follows that
the $L^p$-Poincar\'e inequality for $D_H$ holds for $1<p<\infty$. A similar argument 
(using again that $H_\infty\embed H$) shows that if 
the inclusion $ H\embed E$ is compact, then  $R(\lambda,A_H)$ is compact as well and
the $L^p$-Poincar\'e inequalities for $D_H$ and $D_H\s$ hold for $1<p<\infty$. 

In fact the same results hold if  $\Dom(A^n)\embed H_\infty$ for some large enough $n\ge 1$. We give the 
argument for $n=2$; it is clear from this argument that we may proceed inductively to prove the general case.
For $n=2$ we repeat the above proof we now obtain $\mu R(\mu,A_\infty )R(\lambda,A_\infty) 
= \mu  i_{A^2} R(\mu,A)R(\lambda,A) i_\infty$, where $i_\infty:H_\infty\embed E$ and 
$i_{A^2}: \Dom(A^2)\embed H_\infty$ are the inclusion mappings. It follows that 
$\mu R(\mu,A_\infty )R(\lambda,A_\infty) $ is compact for each $\mu\in\varrho(A_\infty)$.
Passing to the limit $\mu\to\infty$, noting that by the resolvent identity we have
\begin{align*} 
\ & \Big\n \mu R(\mu,A_\infty )R(\lambda,A_\infty) - R(\lambda,A_\infty) \Big\n
\\ & \qquad = \Big\n \frac\mu{\mu-\lambda} (R(\lambda,A_\infty ) - R(\mu,A_\infty))  - R(\lambda,A_\infty) \Big\n
\\ & \qquad \le  \Big\n \frac\mu{\mu-\lambda}R(\mu,A_\infty) \Big\n + \Big\n\Big(\frac\mu{\mu-\lambda}-1 \Big) R(\lambda,A_\infty)\Big\n, 
\end{align*}
and using that $\n R(\nu,A_\infty)\n \le 1/\nu$, 
it follows that $ R(\lambda,A_\infty)$ is compact, being the uniform limit of compact operators.
\end{example}


\section{Compactness results}\label{sec:comp}

In \cite{CMG}, a condition equivalent to the Poincar\'e inequality has been used to prove, under an 
additional Hilbert-Schmidt assumption, the compactness of the semigroup
$P\otimes S_H\s$ on $L^p(E,\mu_\infty;H)$. The importance of this semigroup
is apparent from the proof of Theorem \ref{thm:poncare-DHs} and 
the results in \cite{CMG, ChoGol01, MaaNee09, MaaNee11} where
this semigroup plays a crucial r\^ole in identifying the domains of $\sqrt{-L}$ and $L$.
Here we wish to show that the compactness of this semigroup and its resolvent can be deduced under quite minimal assumptions. 

We begin with a lemma which is based on the classical result of Paley \cite{Pal} and Marcinkiewicz and Zygmund \cite{MarZyg}
(see also \cite{Rub}) 
 that if $T$ is a bounded operator on a space $L^p(\nu)$ and if $H$ is a Hilbert space, 
 then $T\otimes I$ is bounded on $L^p(\nu;H)$ and 
$\n T\otimes I\n  = \n T\n.$ As a direct consequence, if $S$ is a bounded operator on $H$,
then $T\otimes S = (T\otimes I)\circ (I\otimes S)$ is bounded on $L^p(\nu;H)$ and
$\n T\otimes S\n \le \n T\n\n S\n$.

\begin{lemma}\label{lem:compact-TS} Let $1\le p<\infty$.
 If $T$ is compact on $L^p(\nu)$ and $S$ is compact on $H$, then $T\otimes S$ is compact on $L^p(\nu;H)$. 
\end{lemma}
\begin{proof}
Since compactness can be tested sequentially, there is no loss of generality in assuming that both $L^p(\nu)$ and $H$ are separable.
Since separable spaces $L^p(\nu)$ have the approximation property, by \cite[Theorem 1.e.4]{LinTza} there is a finite rank operator $T'$ on $L^p(\nu)$
such that $\n T-T'\n<\eps$. Similarly there is a finite rank operator $S'$ on $H$
such that $\n S-S'\n<\eps$. Then $T'\otimes S'$ is a finite rank operator on $L^p(\nu;H)$
and 
$$ \n T'\otimes S'-T\otimes S\n \le \n T'\otimes (S'-S)\n + \n (T'-T)\otimes S\n
\le  \eps((\n T\n +\eps) + \n S\n).$$
This proves that $T\otimes S$ can be uniformly approximated by finite rank operators.
\end{proof}

We now return to the setting of the previous section. Since a semigroup which is norm continuous
for $t>0$ is compact for $t>0$ if and only if its resolvent operators are compact, Lemma \ref{lem:compact-TS}
implies:

\begin{proposition} Let $1<p<\infty$ and suppose that Assumptions \ref{ass:mu-infty} and \ref{ass:L-analytic} hold. 
 If $P$ has compact resolvent on $L^p(E,\mu_\infty)$, then
$P\otimes S_H\s$ has compact resolvent on $L^p(E,\mu_\infty;H)$. 
\end{proposition}
The generator of $P\otimes S_H\s$ equals $L\otimes I + I\otimes A_H\s$. As we have seen, the compactness of the 
resolvent of $L$ implies the compactness of the resolvent $A_H\s$. 
Thus the proposition suggests the more general problem whether $A\otimes I+I\otimes B$ 
has compact resolvent if $A$ and $B$ have compact resolvents.
Our final result gives an affirmative answer for sectorial operators $A$ and $B$  of angle $<\frac12\pi$. Recall that 
a densely defined closed linear operator $A$ is said to be {\em sectorial operator of angle $<\frac12\pi$}
if there exists an angle $0<\theta<\frac12\pi$ such that $\{|\arg z| > \theta\}\subseteq\varrho(A)$ and
$\sup_{\{|\arg z| > \theta\}}\n z R(z,A)\n < \infty$.

\begin{proposition}\label{prop:A+B} Let $1\le p<\infty$ and suppose that $A$ and $B$ are sectorial operators of 
angle $<\frac12\pi$ on $L^p(\nu)$ and $H$, respectively. 
If, for some $w_0\in\varrho(A)$ and $z_0\in \varrho(A)$, the operators $R(w_0,A)$ and $ R(z_0,B)$ are compact, 
then  $A\otimes I + I \otimes B$ has compact resolvent on $L^p(\nu;H)$.
\end{proposition}
\begin{proof}
Fix numbers $\omega_A<\theta_A <\frac12\pi$, $\omega_B<\theta_B <\frac12\pi$, where $\omega_A$ and $\omega_B$ denote the angles 
of sectoriality of $A$ and $B$. Fix $\lambda\in\C$ with $|\arg \lambda|>\theta$ and fix a number $0<r<|\lambda|$. 
Let $\gamma_{A,r}$ and $\gamma_{B,r}$ 
be the downwards oriented
boundaries of $\{|z|<r\}\cup\{|\arg z|<\theta_A\}$ and  $\{|z|<r\}\cup\{|\arg z|<\theta_B\}$.
It follows from \cite[Formulas (2.2), (2.3)]{KalWei} and a limiting argument that 
 \begin{align}\label{eq:AB} R(\lambda,A\otimes I+B\otimes I) = \frac1{(2\pi i)^2}\int_{\gamma_{B,r}}\int_{\gamma_{A,r}} 
  \frac{1}{\lambda - (w+z)} R(w,A)\otimes R(z,B)\,dw\,dz;
  \end{align}
note that the double integral on the right-hand side converges absolutely.

Given $\e>0$ fix $R>r$ so large that 
$$\Big\n\frac1{(2\pi i)^2}\int_{\gamma_{B,r}\cap \complement B_R}\int_{\gamma_{A,r}\cap \complement B_R} 
 \frac{1}{\lambda - (w+z)} R(w,A)\otimes R(z,B)\,dw\,dz \Big\n < \e,
$$ where $B_R = \{z\in \C: |z|<\R\}$ and $\complement B_R$ is its complement.
By Lemma \ref{lem:compact-TS} and \cite[Theorem 1.3]{Voi} the operator 
$$\frac1{(2\pi i)^2}\int_{\gamma_{B,r}\cap  B_R}\int_{\gamma_{A,r}\cap B_R} 
 \frac{1}{\lambda - (w+z)} R(w,A)\otimes R(z,B)\,dw\,dz$$
 is compact, as it is the strong integral over a finite measure space of an integrand with values in the space of
 compact operators.
As a consequence, for each $\e>0$ we obtain that $R(\lambda,A\otimes +I\otimes B) = K_\e + L_\e$ with $K_\e$ compact and 
$L_\e$ bounded with $\n L_\e \n < \e$. It follows that the range of the unit ball of $L^p(\nu;H)$ under 
$R(\lambda,A\otimes I+I\otimes B)$ is totally bounded and therefore relatively compact.
\end{proof}
The formula \eqref{eq:AB} for the resolvent of the sum of two operators goes back to Bianchi and Favella \cite{BiaFav} who considered bounded $A$ and $B$.
It can be viewed as a special instance of the so-called joint functional 
calculus for sectorial operators; see \cite[Theorem 2.2]{LLM}, \cite[Theorem 12.12]{KunWei}.

\begin{remark}
 The above proof easily extends to tensor products of $C_0$-semigroups on arbitrary Banach spaces, provided one makes 
 appropriate assumptions on the boundedness of the tensor products of the various bounded operators
 involved.
\end{remark}

\begin{remark}
 The same proof may be used to see that if $A$ and $B$ are resolvent commuting sectorial operators of 
angle $<\frac12\pi$ on a Banach space $X$ and if,
for some $w_0\in\varrho(A)$ and $z_0\in \varrho(A)$, the operator $R(w_0,A) R(z_0,B)$ is compact
on $X$, then  $A+ B$ has compact resolvent on $X$.
\end{remark}

\noindent{\em Acknowledgment} -- I thank Ben Goldys and Jan Maas for providing helpful comments
and the anonymous referee for suggesting some improvements.

\bibliographystyle{plain}
\bibliography{references-poincare}

\end{document}